\documentclass[12pt]{amsart}
\usepackage{amsfonts}
\usepackage{amssymb}
\usepackage{amsmath}
\usepackage{amsfonts}
\usepackage{graphicx}
\usepackage{amscd}
\usepackage{rotating}
\usepackage{amsmath,amsfonts,amssymb,amsthm,mathrsfs,xypic}
\xyoption{curve}
\usepackage{fancybox}
\usepackage{color}
\usepackage[perpage,symbol]{footmisc}
\setfnsymbol{wiley}

\usepackage{amsfonts}
\usepackage{graphicx}
\usepackage{amscd}
\usepackage[all]{xy}
\usepackage{amsmath,accents,amsfonts,amssymb,amsthm,mathrsfs,xypic,extarrows,mathtools}
\usepackage{leftidx}
\newtheorem{proposition}{Proposition}[section]
\newtheorem{lemma}[proposition]{Lemma}
\newtheorem{corollary}[proposition]{Corollary}
\newtheorem{theorem}[proposition]{Theorem}

\theoremstyle{definition}
\newtheorem{definition}[proposition]{Definition}

\theoremstyle{remark}
\newtheorem{remark}[proposition]{Remark}

\newcommand{\Ext}{\mbox{\rm Ext}}
\newcommand{\Hom}{\mbox{\rm Hom}}
\newcommand{\RHom}{\mbox{\rm RHom}}

\newcommand{\im}{\mbox{\rm Im}}
\newcommand{\id}{\mbox{\rm id}}
\newcommand{\Coker}{\mbox{\rm Coker}}
\newcommand{\Cone}{\mbox{\rm Cone}}
\newcommand{\Ker}{\mbox{\rm Ker}}
\newcommand{\DP}{\mbox{\rm DP}}
\newcommand{\Dpd}{\mbox{\rm Dpd}}
\newcommand{\Gpd}{\mbox{\rm Gpd}}
\newcommand{\pd}{\mbox{\rm pd}}
\newcommand{\fd}{\mbox{\rm fd}}

\newcommand{\h}{\mbox{\rm H}}
\newcommand{\s}{\mbox{\rm S}}
\newcommand{\Z}{\mbox{\rm Z}}
\newcommand{\DG}{\mbox{\rm DG}}
\newcommand{\B}{\mbox{\rm B}}
\newcommand{\C}{\mbox{\rm C}}
\newcommand{\p}{\mbox{\rm P}}

\begin{document}

\title{DING PROJECTIVE DIMENSION OF COMPLEXES*}
\author{ZHANPING WANG \ \ \ \ ZHONGKUI LIU}

\footnote[0]{*Supported by National Natural Science Foundation of China (Grant No. 11201377, 11261050) and Program of Science and Technique of Gansu Province (Grant No. 1208RJZA145).}\footnote[0]{Address
correspondence to Zhanping Wang, Department of Mathematics, Northwest Normal University, Lanzhou 730070, PR China.}\footnote[0]{E-mail: wangzp@nwnu.edu.cn (Z.P. Wang),
liuzk@nwnu.edu.cn (Z.K. Liu).}

\date{}\maketitle

\noindent{\footnotesize {\bf Abstract} In this paper, we define and study a notion of Ding projective dimension for complexes of left modules over associative rings. In particular, we consider the class of homologically bounded below complexes of left $R$-modules, and show that Ding projective dimension has a nice functorial description.

\vspace{0.2cm}
\noindent{\footnotesize {2010 {\bf{Mathematics Subject
Classification}}:}
18G35, 55U15, 13D05, 16E30

\noindent{\footnotesize {{\bf{Key words}}:} Ding projective dimension,  Ding projective modules, complexes.

\section{Introduction}
In \cite{Avramov1991}, Avramov and Foxby defined the projective (resp. injective or flat) dimension for unbounded complexes by means of DG-projective (resp. DG-injective or DG-flat) resolutions. A complex $P$ of $R$-modules is called DG-projective if $\Hom _{R}(P, -)$ transforms surjective quasi-isomorphisms into surjective quasi-isomorphisms, which is equivalent to saying that $P$ is a complex of projective $R$-modules and $\Hom _{R}(P, X)$ is exact for every exact complex $X$ by \cite[1.2.P]{Avramov1991}. A DG-projective resolution of $X$ is a quasi-isomorphism
$P\rightarrow X$ with $P$ DG-projective. By \cite[Corollary 3.10]{Enochs1996}, every complex has a surjective DG-projective resolution $P\rightarrow X$. If $X$ is homologically bounded below, then $P$ can be chosen so that $\inf\{i\mid
P_{i}\neq0\}=\inf X$.

Over commutative local rings, Yassemi \cite{Yassemi1995} and Christensen \cite{Christensen2000} introduced a Gorenstein projective dimension for complexes with bounded below homology. In \cite{Veliche2006}, Veliche defined and studied Gorenstein projective dimension for complexes of left $R$-modules over associative ring $R$. Not much later Gorenstein injective and Gorenstein flat dimension for complexes were introduced and studied in \cite{Asadollahi2006, Iacob2009}. These Gorenstein dimensions are related to the Gorenstein rings. General background materials about Gorenstein homological algebra can be found in \cite{Enochs2000, Enochs2001, Rozas1999}.

In \cite{Ding2009}, Ding, Li and Mao introduced and studied strongly Gorenstein flat modules, and several well-known classes of rings are characterized in terms of these modules. A left $R$-module $M$ is called strongly Gorenstein flat if there is an exact sequence
\[\cdots \longrightarrow P_{1}\longrightarrow P_{0}\longrightarrow P_{-1}\longrightarrow P_{-2}\longrightarrow \cdots\]
of projective left $R$-modules with $M=\Coker (P_{0}\longrightarrow P_{-1})$ such that $\Hom (-, \mathcal{F})$ leaves the sequence exact, where $\mathcal{F}$ stands for the class of all flat left $R$-modules. Since strongly Gorenstein flat modules have properties analogous to Gorenstein projective modules, Gillespie \cite{Gillespie2010} called these modules Ding projective modules. For every left $R$-module $M$ over an associative ring $R$, Ding at al. also defined and investigated the strongly Gorenstein flat dimension for modules and rings.

The main purpose of this paper is to introduce and study a concept of Ding projective dimension $\Dpd_{R}(X)$ associated to every complex $X$ of left $R$-modules over an arbitrary associative ring $R$. In particular, we consider the class of homologically bounded below complexes of left $R$-modules, and show that Ding projective dimension has a nice functorial description.

In this paper, $R$ denotes an associative ring with unity. We consistently use the notation from the appendix of \cite{Christensen2000}. In particular, the category of $R$-complexes is denoted $\C(R)$, we use subscripts $\sqsubset$, $\sqsupset$ and $\square$ to denote boundedness conditions, and use subscripts $(\sqsubset)$, $(\sqsupset)$ and $(\square)$ to denote homological boundedness conditions. For example, $\C_{\sqsupset}(R)$ is the full subcategory of $\C(R)$ of bounded below complexes; $\C_{(\sqsupset)}(R)$ is the full subcategory of $\C(R)$ of homologically bounded below complexes.
Given a complex $C$ and an integer $i$, $\Sigma^{i}C$ denotes the complex
such that $(\Sigma^{i}C)_{n}=C_{n-i}$ and whose boundary operators are
$(-1)^{i}\delta_{n-i}^{C}$; The $n$th homology module of $C$ is the module $\h_{n}(C)=\Z_{n}(C)/\B_{n}(C)$, where $\Z_{n}(C)=\Ker (\delta_{n}^{C}), ~~ \B_{n}(C)=\im (\delta_{n+1}^{C})$; we set $\h^{n}(C)=\h_{-n}(C), \ \C_{n}(C)=\Coker (\delta_{n+1}^{C}).$ Given a left
$R$-module $M$, we will denote by $\s^{n}(M)$ the
complex with $M$ in the $n$th place and $0$ in the other places. For more details of complexes used in this paper the reader can consult \cite{Hartshorne1966, MacLane1998}.

\section{Ding projective dimension of complexes}
In the section, $\mathcal{F}$ stands for the class of flat modules.

\begin{definition} A complex of $R$-modules $T$ is said to be totally $\mathcal{F}$-acyclic if the following conditions are satisfied:

(1) $T_{n}$ is projective for every $n\in \mathbb{Z}$.

(2) $T$ is exact.

(3) $\Hom_{R}(T, F)$ is exact for every $R$-module $F\in \mathcal{F}$.
\end{definition}

An exact complex of projective $R$-modules $T$ is said to be totally acyclic \cite{Veliche2006} if $\Hom_{R}(T, P)$ is exact for every projective $R$-module $P$. By definitions, totally $\mathcal{F}$-acyclic complex is totally acyclic.

For totally $\mathcal{F}$-acyclic complex, we have the following two properties using the routine proof.

\begin{lemma}\label{lem2.2} Let $T$ be a totally $\mathcal{F}$-acyclic complex.
If $Q$ is a complex of flat modules and $n$ is an integer,
then any morphism of complexes $\underline{\varphi}: T_{n}\sqsupset~\longrightarrow
Q_{n}\sqsupset$ can be extended to morphism $\varphi: T\longrightarrow
Q$ such that $\varphi_{n}\sqsupset~=\underline{\varphi}$.
Every morphism $\varphi$ with this property is defined as unique up to homotopy.
\end{lemma}

\begin{lemma}\label{lem2.3} Let $T$ be a totally $\mathcal{F}$-acyclic complex. If $Q$ is a bounded above complex of flat modules, then
\[\h(\Hom_{R}(T, Q))=0.\]
\end{lemma}

An $R$-module $M$ is called strongly Gorenstein flat \cite{Ding2009} if there exists a totally $\mathcal{F}$-acyclic complex $T$ such that $\C_{0}(T)=M$. Since strongly Gorenstein flat modules have properties analogous to Gorenstein projective modules, Gillespie \cite{Gillespie2010} call these modules Ding projective modules. Note that every projective module is Ding projective, and every cokernel $
\C_{n}(T)$ of totally $\mathcal{F}$-acyclic complex $T$ is Ding projective.

Ding projective modules have the following properties.

\begin{lemma}\label{lem2.4} Let $\DP(R)$ stand for the class of Ding projective modules. The following assertions hold.

$\mathrm{(1)}$ If $M\in \DP(R)$, then $\Ext^{i}_{R}(M,L)=0$ for all $i>0$ and all module $L$ of finite flat or finite injective dimension.

$\mathrm{(2)}$ $\DP(R)$ is a projectively resolving class, and closed under direct sums and direct summands.

$\mathrm{(3)}$ If $0\rightarrow A \rightarrow B \rightarrow C\rightarrow 0$ is an exact sequence, and $A, ~~B\in \DP(R)$, $\Ext_{R}^{1}(C,F)=0$ for every flat module $F$, then $C\in \DP(R)$.
\end{lemma}

\begin{proof} (1) is obvious.

(2) It follows by analogy with the proof of Theorem 2.5 in \cite{Holm2004}.

(3) It follows by analogy with the proof of Corollary 2.11 in \cite{Holm2004}.

\end{proof}

\begin{lemma}\label{lem2.5} Let $T$ be an exact complex of projective modules. Then the following are equivalent:

$\mathrm{(1)}$ $T$ is totally $\mathcal{F}$-acyclic.

$\mathrm{(2)}$ $\C_{i}(T)$ is Ding projective for all $i\in \mathbb{Z}$.

$\mathrm{(3)}$ $\C_{i}(T)$ is Ding projective for infinitely many $i\leq 0$.
\end{lemma}

\begin{proof} The implications $(1)\Rightarrow (2)\Rightarrow (3)$ are clear, so we only argue $(3)\Rightarrow (1)$.

Let $Q$ be a flat $R$-module and fix $n\in \mathbb{Z}$. We need to show $\h^{n}(\Hom_{R}(T, Q))=0$. By $(3)$, we choose an integer $m\geq1$ such that $n-m$ is small enough for $\C_{n-m}(T)$ to be Ding projective.
Thus
\[\h^{n}(\Hom_{R}(T, Q))=\Ext_{R}^{m}(\C_{n-m}(T), Q)=0.\]
\end{proof}

\begin{lemma} \label{lem2.6} If $G$ is Ding projective and $\cdots\longrightarrow P_{2}\longrightarrow P_{1}\longrightarrow P_{0}\stackrel{\alpha}\longrightarrow G\longrightarrow 0$ is a projective resolution of $G$, then there exists a totally $\mathcal{F}$-acyclic complex $T$ such that $T_{0}\sqsupset~=P$, where $P=\cdots\longrightarrow P_{2}\longrightarrow P_{1}\longrightarrow P_{0}\longrightarrow 0\longrightarrow \cdots.$
\end{lemma}

\begin{proof} By the definition of Ding projective modules, there is  a totally $\mathcal{F}$-acyclic complex $T'$ such that $\C_{0}(T')=G.$ Set
\[T_{i}=\bigg\{
\begin{aligned}
&T_{i}'\ for\ i<0,\\
&P_{i}\ for \ i\geq0,\\
\end{aligned} \]
and
\[\delta_{i}^{T}=\left\{
\begin{aligned}
\delta_{i}^{T'} &\  for \ i<0,\\
\beta\alpha &\ for \ i=0,\\
\delta_{i}^{P}&\ for \ i>0.
\end{aligned}\right. \]
where $\alpha: P_{0}\longrightarrow G$ and $\beta: G\longrightarrow T'_{-1}$ are the canonical maps. The complex $T$ is exact, $\C_{0}(T)=G$ and $\sqsubset_{-1}T=~\sqsubset_{-1}T'$, so $T$ is totally $\mathcal{F}$-acyclic by Lemma \ref{lem2.5}.
\end{proof}

According to \cite{Ding2009}, Ding projective dimension, or strongly Gorenstein flat dimension, of $M$ is defined by:
\[Dpd(M)=\inf \left\{n\in \mathbb{N}_{0}~\bigg{|}~
0\longrightarrow G_{n}\longrightarrow \cdots \longrightarrow G_{1}\longrightarrow G_{0}\longrightarrow M\longrightarrow 0 ~~is~~ exact,~~and~~G_{i}\in \DP(R)\right\}.\]

\begin{lemma}\label{lem2.7} $\mathrm{(1)}$ Let $M$ be an $R$-module with finite Ding projective dimension, and let $n$ be an integer. Then the following conditions are equivalent:

$\mathrm{(i)}$ $Dpd(M)\leq n$.

$\mathrm{(ii)}$ $\Ext_{R}^{i}(M, L)=0$ for all $i>n$, and all $R$-modules $L$ of finite flat dimension.

$\mathrm{(iii)}$ $\Ext_{R}^{i}(M, F)=0$ for all $i>n$, and all flat $R$-modules $F$.

$\mathrm{(iV)}$ For every exact sequence $0\longrightarrow K_{n} \longrightarrow
G_{n-1}\longrightarrow\cdots \longrightarrow G_{1}\longrightarrow
G_{0}\longrightarrow M\longrightarrow 0$,
where $G_{i}$ is Ding projective,
then also $K_{n}$ is Ding projective.

$\mathrm{(2)}$ If $(M_{i})_{i\in I}$ is a family of $R$-modules, then
\[Dpd(\oplus_{I} M_{i})=\sup \{Dpd(M_{i})|i\in I\}.\]
\end{lemma}

\begin{proof} They follow by analogy with the proofs of Theorem 2.20 and Proposition 2.19 respectively in \cite{Holm2004}.
\end{proof}

\begin{definition} An $\mathcal{F}$-complete resolution of $X$ is a diagram of morphisms of complexes $T\stackrel{\tau}\longrightarrow P\stackrel{\pi}\longrightarrow X$, where $\pi: P\longrightarrow X$ is a DG-projective resolution, $T$ is a totally $\mathcal{F}$-acyclic complex and $\tau_{i}$ is bijective for all $i\gg0$. An $\mathcal{F}$-complete resolution $T\stackrel{\tau}\longrightarrow P\stackrel{\pi}\longrightarrow X$ of $X$ is said to be surjective if $\tau_{i}$ is surjective for all $i\in \mathbb{Z}$.
\end{definition}

\begin{lemma} \label{lem2.9} Let $T\stackrel{\tau}\longrightarrow P\stackrel{\pi}\longrightarrow X$ be an $\mathcal{F}$-complete resolution. If $g$ is an integer such that $\tau_{i}$ is bijective for all $i\geq g$, then there exists an $\mathcal{F}$-complete resolution $T'\stackrel{\tau'}\longrightarrow P\stackrel{\pi}\longrightarrow X$ such that $\tau'$ is a surjective morphism, $\tau_{i}'$ is bijective for all $i\geq g$, and a homology equivalence $\alpha: T\longrightarrow T'$ such that $\tau=\tau'\alpha$ and $\alpha_{i}=\id^{T_{i}}$ for all $i\geq g$.
\end{lemma}

\begin{proof} Set $(T')_{n}=(T\oplus ~\sqsubset_{g-1}P\oplus\Sigma^{-1}~\sqsubset_{g-1}P)_{n}$ as desired.
\end{proof}

\begin{definition} The Ding projective dimension of $X$ is defined by
\[\Dpd_{R}(X)=\inf \left\{n\in \mathbb{Z}~\rule[-6mm]{0.2mm}{15mm}~
\begin{aligned}
&T\stackrel{\tau}\longrightarrow P\stackrel{\pi}\longrightarrow X ~\text{is an $\mathcal{F}$-complete resolution}  \\
  & \text{with}~\tau_{i}: T_{i}\rightarrow P_{i}\ \text{bijective for each}\ i\geq n \ \\
 \end{aligned}\right\}.\]
 \end{definition}

 \begin{remark} (1) For any complex $X$, $\Dpd_{R}(X)\in \{-\infty\}\cup \mathbb{Z}\cup\{\infty\}.$

 (2) $\Dpd_{R}(X)=-\infty$ if and only if $X$ is exact.

 (3) For any $k\in \mathbb{Z}$, $\Dpd_{R}(\Sigma^{k}X)=\Dpd_{R}(X)+k.$
\end{remark}

\begin{theorem}\label{the2.16} Let $X$ be a complex, $n$ an integer. Then the following assertions are equivalent:

$\mathrm{(1)}$ $\Dpd_{R}(X)\leq n$.

$\mathrm{(2)}$ $\sup X\leq n$ and there exists a $\DG$-projective resolution $P\longrightarrow X$ such that the module $\C_{n}(P)$ is Ding projective.

$\mathrm{(3)}$ $\sup X\leq n$ and for every $\DG$-projective resolution $P'\longrightarrow X$, the module $\C_{n}(P')$ is Ding projective.

$\mathrm{(4)}$ For every $\DG$-projective resolution $P'\longrightarrow X$, there exists a surjective $\mathcal{F}$-complete resolution $T'\longrightarrow P'\longrightarrow X$ such that $\tau_{i}'=\id^{T_{i}'}$ for all $i\geq n$.
\end{theorem}

\begin{proof} $(1)\Rightarrow (2)$ By hypothesis,
 there exists an $\mathcal{F}$-complete resolution $T\stackrel{\tau}\longrightarrow
P\stackrel{\pi}\longrightarrow X$ such that $\tau_{n}\sqsupset~: T_{n}\sqsupset~\longrightarrow P_{n}\sqsupset$ is an isomorphism of complexes. This yields isomorphisms $\h_{i}(X)\cong \h_{i}(P)$ for all $i\in
\mathbb{Z}$, $\h_{i}(P)\cong \h_{i}(T)$ for all $i>n$,
and $\C_{n}(P)\cong \C_{n}(T)$.
Since the complex $T$ is totally $\mathcal{F}$-acyclic, we have $\h_{i}(T)=0$ for each $i\in \mathbb{Z}$ and $\C_{n}(T)$ is Ding projective.

$(2)\Rightarrow (3)$ Let $P'\longrightarrow X$ be a DG-projective resolution. Then $P\simeq P'$. Since $P'$ is DG-projective, there exists a quasi-isomorphism $P'\longrightarrow P$. We can assume that $P'\longrightarrow P$ is a surjective quasi-isomorphism (~if not, let let $F\longrightarrow P$ be surjective with $F$ a projective complex, then $F\oplus P'\longrightarrow P$ is a surjective quasi-isomorphism~). Hence there exists an exact sequence
\[0\longrightarrow K\longrightarrow P'\longrightarrow P\longrightarrow 0\]
with $K$ an exact complex. Both $P'$ and $P$ are DG-projective complexes, so $K$ is a DG-projective complex.
Thus $K$ is exact and DG-projective, and so $K$ is a projective complex. In addition, we have an exact sequence
\[0\longrightarrow \C_{n}(K)\longrightarrow \C_{n}(P')\longrightarrow \C_{n}(P)\longrightarrow 0.\]
By (2), we get that $\C_{n}(P)$ is Ding projective. But $\C_{n}(K)$ projective, and so $\C_{n}(K)$ is Ding projective. It follows that $\C_{n}(P')$ is Ding projective by Lemma \ref{lem2.4}.

$(3)\Rightarrow (4)$ Let $P'\longrightarrow X$ be a DG-projective resolution with $\C_{n}(P')$ Ding projective and $\h_{i}(P')=0$ for all $i>n$. Then $\Sigma^{-n}P'_{n}\sqsupset~\longrightarrow \C_{n}(P')$ is a projective resolution.  By Lemma \ref{lem2.6}, there is a totally $\mathcal{F}$-acyclic complex $T''$ such that $T''_{n}\sqsupset~=P'_{n}\sqsupset$. So we obtain an $\mathcal{F}$-complete resolution $T''\stackrel{\tau''}\longrightarrow P'\stackrel{\pi'}\longrightarrow X$ with $\tau_{i}''=\id^{T_{i}''}$ for all $i\geq n$ and $\C_{n}(T'')\cong \C_{n}(P')$ by Lemma \ref{lem2.2}. From Lemma \ref{lem2.9}, we get a surjective $\mathcal{F}$-complete resolution $T'\longrightarrow P'\longrightarrow X$ with the desired properties.

$(4)\Rightarrow (1)$ is clear.
\end{proof}
\begin{corollary} For every family of complexes of $R$-modules $(X_{i})_{i\in I}$ one has \[\Dpd_{R}(\oplus_{I} X_{i})=\sup \{\Dpd_{R}(X_{i})|i\in I\}.\]
\end{corollary}
\begin{proof} For each $i\in I$, there is a DG-projective resolution $P_{i}\longrightarrow X_{i}$. Set $P=\oplus_{I} P_{i}$.  Then $P\longrightarrow \oplus_{I} X_{i}$ is a DG-projective resolution and $\C_{n}(P)=\oplus_{I} \C_{n}(P_{i})$ for each $n\in \mathbb{Z}$. Thus the assertion follows from Theorem \ref{the2.16} and Lemma \ref{lem2.7}.
\end{proof}
\begin{corollary} Let $M$ be an $R$-module. Then
~$\Dpd_{R}(\s^{0}(M))
=\Dpd(M)$.
\end{corollary}
\begin{proof} Let
\[\cdots\longrightarrow P_{2}\longrightarrow P_{1}\longrightarrow P_{0}\longrightarrow M\longrightarrow 0\]
be a projective resolution of $M$. Then $P\longrightarrow \s^{0}(M)$ is a DG-projective resolution, where $P=\cdots\longrightarrow P_{2}\longrightarrow P_{1}\longrightarrow P_{0}\longrightarrow0\longrightarrow\cdots$. If $\Dpd(M)=\infty$, and $\Dpd_{R}(\s^{0}(M))=l<\infty$, then $\C_{j}(P)$ is Ding projective for any $j\geq l$ by Theorem \ref{the2.16}.
Since
\[0\longrightarrow \C_{l}(P)\longrightarrow P_{l-1}\longrightarrow \cdots\longrightarrow P_{1}\longrightarrow P_{0}\longrightarrow M\longrightarrow 0\]
is exact with $\C_{l}(P)$ Ding projective and $P_{j}$ projective modules, it follows that $\Dpd(M)\leq l$.  This contradicts $\Dpd(M)=\infty$. So $\Dpd_{R}(\s^{0}(M))=\infty$.
If $\Dpd(M)=l<\infty$, then $\C_{l}(P)$ is Ding projective, and so $\C_{j}(P)$ is Ding projective for all $j\geq l$ by Lemma \ref{lem2.4}. Hence $P\rightarrow \s^{0}(M)$ is a DG-projective resolution with $\C_{j}(P)$ Ding projective and $\h_{j}(P)=0$ for all $j\geq l$. By Theorem \ref{the2.16}, $\Dpd_{R}(\s^{0}(M))\leq l$. Suppose that $\Dpd_{R}(\s^{0}(M))\leq l-1$. Then $C_{l-1}(P)$ is Ding projective. In the exact sequence \[0\longrightarrow C_{l-1}(P)\longrightarrow P_{l-2}\longrightarrow \cdots\longrightarrow P_{1}\longrightarrow P_{0}\longrightarrow M\longrightarrow 0,\]
$\C_{l-1}(P)$ is Ding projective and
every $P_{j}$ is projective, which yields $\Dpd(M)\leq l-1$. This contradicts $\Dpd(M)=l$. Therefore, $\Dpd_{R}(\s ^{0}(M))=l$.
\end{proof}

\begin{corollary} Let $X$ be a complex. Then $\Gpd_{R}(X)\leq \Dpd_{R}(X)\leq \pd_{R}(X)$, with equalities if $\pd_{R}(X)$ is finite.
\end{corollary}
\begin{proof} If $\pd_{R}(X)=\infty$, then it is clear.
If $\pd_{R}(X)=-\infty$, then $X$ is exact, so
$\Gpd_{R}(X)=\Dpd_{R}(X)=-\infty$.
Let $\pd_{R}(X)=g<\infty$. Then for any DG-projective resolution $P\longrightarrow
X$ we have $\sup P\leq g$ and $\C_{j}(P)$ is projective for all $j\geq g$.
So $\C_{j}(P)$ is Ding projective for all $j\geq g$. By Theorem \ref{the2.16}, we get $\Dpd_{R}(X)\leq
g$. In similar method, we obtain that $\Gpd_{R}(X)\leq k$ if $\Dpd_{R}(X)=k$. The final assertion follows from Theorem 3.7 in \cite{Veliche2006}.
\end{proof}

\begin{proposition}Let $0\longrightarrow X\longrightarrow Y\longrightarrow Z\longrightarrow 0$ be an exact sequence of complexes. If two complexes have finite Ding projective dimension, then so does the third.
\end{proposition}
\begin{proof} By \cite[Proposition 1.3.8]{Veliche2006},
there is an exact sequence of complexes $0\longrightarrow P^{X}\longrightarrow
P^{Y}\longrightarrow P^{Z}\longrightarrow
0$ with $P^{X}\longrightarrow X$, $P^{Y}\longrightarrow
Y$ and $P^{Z}\longrightarrow Z$ DG-projective resolutions. If two of the complexes $X, Y,
Z$ have finite Ding projective dimension,
 then there is $n\in \mathbb{Z}$ such that $\h_{j}(P^{X})=\h_{j}(P^{Y})=\h_{j}(P^{Z})=0$ for all $j\geq n$.  For each $j\geq n$,
we have an exact sequence
\[0\longrightarrow
\C_{j}(P^{X})\longrightarrow \C_{j}(P^{Y})\longrightarrow
\C_{j}(P^{Z})\longrightarrow 0\]
in $R$-Mod. If $\C_{j}(P^{Z})$ is Ding projective, then $\C_{j}(P^{X})$ is Ding projective if and only if $\C_{j}(P^{Y})$ is Ding projective.
If $\C_{j}(P^{X})$ and $\C_{j}(P^{Y})$ are Ding projective, then $\Dpd(\C_{j}(P^{Z}))\leq 1$, and so $\C_{j+1}(P^{Z})$ is Ding projective by Lemma \ref{lem2.7}. Therefore the assertion follows from Theorem \ref{the2.16}.
\end{proof}

\begin{proposition} \label{prop2.22} Let $X$ be a homologically bounded below complex. Then
\[\Dpd_{R}(X)=\inf\left\{\sup\left\{l\in \mathbb{Z}~\bigg{|}~
Q_{l}\neq0\right\}~\rule[-6mm]{0.2mm}{15mm}~
\begin{aligned}
&Q\simeq X ~~\text{and}~~Q ~~\text{is a bounded below}\\
& \text{complex of Ding projective modules}\\
  \end{aligned}\right\}
.\]
\end{proposition}
\begin{proof} Since $X$ is a homologically bounded below complex, we can assume $\inf X=0$. Let $P\rightarrow X$ be a DG-projective resolution of $X$ with $\inf\{i\mid
P_{i}\neq0\}=0$.
Set \[\Omega=\inf\left\{\sup\left\{l\in \mathbb{Z}~\bigg{|}~
Q_{l}\neq0\right\}~\rule[-6mm]{0.2mm}{15mm}~
\begin{aligned}
&Q\simeq X ~~\text{and}~~Q ~~\text{is a bounded below}\\
& \text{complex of Ding projective modules}\\
  \end{aligned}\right\}
.\]
If $\Dpd_{R}(X)=n$, then $\C_{j}(P)$ is Ding projective for every $j\geq n$ by Theorem \ref{the2.16}.
Let
\[P'=0\longrightarrow \C_{n}(P)\longrightarrow
P_{n-1}\longrightarrow \cdots\longrightarrow P_{0}\longrightarrow
0,\] and
\[X'=0\longrightarrow \C_{n}(X)\longrightarrow X_{n-1}\longrightarrow X_{n-2}\longrightarrow\cdots.\]
Since $P\simeq X$, we have $P\simeq X'$. From $P'\simeq X'$ and $X\simeq X'$, we get $P'\simeq X$. Each component of $P'$ is a Ding projective module, which implies that $\Omega\leq n$.

Now suppose that $\Omega=m<\infty$.
We are going to show that $\Dpd_{R}(X)\leq m$.
By hypothesis, there exists a complex
\[Q=0\longrightarrow Q_{m}\longrightarrow Q_{m-1}\longrightarrow\cdots\longrightarrow Q_{0}\longrightarrow 0\]
of Ding projective modules such that $Q\simeq X$. Since $P\simeq X\simeq Q$ and $P$ is a DG-projective complex, there is a quasi-isomorphism $P\longrightarrow Q$. In addition, $Q$ is bounded below, so there is a surjective morphism $P^{*}\longrightarrow Q$ with $P^{*}$ a bounded below projective complex. Then $P\oplus P^{*}\longrightarrow Q$ is a surjective quasi-isomorphism. Thus we have an exact sequence
\[0\longrightarrow K\longrightarrow P\oplus P^{*}\longrightarrow Q\longrightarrow 0,\]
 with $K$ exact, which implies that there is an exact sequence
\[0\longrightarrow K_{j}\longrightarrow P_{j}\oplus P^{*}_{j}\longrightarrow Q_{j}\longrightarrow 0\]
in $R$-Mod for all $j\in \mathbb{Z}$. Since $P_{j}\oplus P^{*}_{j}$ and $Q_{j}$ are Ding projective modules, it follows that $K_{j}$ is a Ding projective module. Thus $K$ is a bounded below exact complex of Ding projective modules, and so $\C_{j}(K)$ is a Ding projective module for each $j\in \mathbb{Z}$. In exact sequence
\[0\longrightarrow \C_{m}(K)\longrightarrow \C_{m}(P)\oplus \C_{m}(P^{*})\longrightarrow \C_{m}(Q)\longrightarrow 0,\]
$\C_{m}(Q)=Q_{m}$ and $\C_{m}(K)$ are Ding projective, so $\C_{m}(P)\oplus \C_{m}(P^{*})$ is Ding projective. By Lemma \ref{lem2.4}, $\C_{m}(P)$ is Ding projective. Since $P\longrightarrow X$ is a DG-projective resolution with $\sup P\leq m$ and $\C_{j}(P)$ Ding projective for all $j\geq m$, it follows that $\Dpd_{R}(X)\leq m$. By the above,
$\Dpd_{R}(X)=\infty$ if and only if $\Omega=\infty$; and note that
$\Dpd_{R}(X)=-\infty$ if and only if $X$ is exact if and only if
$\Omega=-\infty$.
\end{proof}

\begin{lemma}\label{lem2.23}

$\mathrm{(1)}$(\cite[Proposition 2.6(a)]{Christensen2006}) Let $\mathcal{U}$ be a class of $R$-modules, and $\alpha: X\rightarrow Y$ be a morphism in $\C(R)$ such that
\[\Hom_{R}(U, \alpha): \Hom_{R}(U, X)\stackrel{\simeq} \longrightarrow \Hom_{R}(U, Y)\]
is a quasi-isomorphism for every module $U\in \mathcal{U}$. If $\widetilde{U}\in \C_{\sqsupset}(R)$ is a complex consisting of modules from $\mathcal{U}$, then the induced morphism $\Hom_{R}(\widetilde{U}, \alpha): \Hom_{R}(\widetilde{U}, X)\stackrel{\simeq}\longrightarrow \Hom_{R}(\widetilde{U}, Y)$ is a quasi-isomorphism.

$\mathrm{(2)}$ (\cite[Proposition 2.7(a)]{Christensen2006}) Let $\mathcal{V}$ be a class of $R$-modules, and $\alpha: X\rightarrow Y$ be a morphism in $\C(R)$ such that
\[\Hom_{R}(\alpha, V): \Hom_{R}(X, V)\stackrel{\simeq}\longrightarrow \Hom_{R}(X, V)\]
is a quasi-isomorphism for every module $V\in \mathcal{V}$. If $\widetilde{V}\in \C_{\sqsubset}(R)$ is a complex consisting of modules from $\mathcal{U}$, then the induced morphism $\Hom_{R}(\alpha, \widetilde{V}): \Hom_{R}(X, \widetilde{V})\stackrel{\simeq}\longrightarrow \Hom_{R}(Y, \widetilde{V})$ is a quasi-isomorphism.
\end{lemma}

\begin{lemma}\label{lem2.24} Let $V\stackrel{\simeq}\rightarrow W$ be a quasi-isomorphism between $R$-complexes, where $V,W\in \C_{\sqsubset}(R)$ and each module in $V$ and $W$ has finite flat or finite injective dimension. If $A\in C_{\sqsupset}(R)$ is a complex of Ding projective modules, then the induced morphism $\Hom_{R}(A, V)\rightarrow \Hom_{R}(A, W)$ is a quasi-isomorphism.
\end{lemma}
\begin{proof}By Lemma \ref{lem2.23}(1), we may immediately reduce to the case, where $A$ is a Ding projective module. In this case we have quasi-isomorphisms $\alpha: P\stackrel{\simeq}\longrightarrow A$ and $\beta: A\stackrel{\simeq}\longrightarrow \widetilde{P}$ in $\C(R)$, where $P\in \C_{\sqsupset}(R)$ and $\widetilde{P}\in \C_{\sqsubset}(R)$ are respectively the left half and right half of a totally $\mathcal{F}$-acyclic complex of $A$. Let $T$ be any $R$-module of finite flat or finite injective dimension. Lemma \ref{lem2.4}(1) implies that a totally $\mathcal{F}$-acyclic complex stays exact when the functor $\Hom_{R}(-, T)$ is applied to it. In particular, the induced morphisms
\[\Hom_{R}(\alpha, T): \Hom_{R}(A, T)\stackrel{\simeq}\longrightarrow \Hom_{R}(P, T),\]
and
\[\Hom_{R}(\beta, T): \Hom_{R}(\widetilde{P}, T)\stackrel{\simeq}\longrightarrow \Hom_{R}(A, T)\]
are quasi-isomorphisms. By Lemma \ref{lem2.23}(2) it follows that $\Hom_{R}(\alpha, V)$ and $\Hom_{R}(\alpha, W)$ are quasi-isomorphisms. In the commutative diagram
\[\xymatrix{\Hom_{R}(A, V)\ar[r]\ar[d]_{\Hom_{R}(\alpha, V)}^{\simeq}&\Hom_{R}(A, W)\ar[d]^{\Hom_{R}(\alpha, V)}_{\simeq}\\
\Hom_{R}(P, V)\ar[r]^{\simeq}&\Hom_{R}(P, W)
}\]
the lower horizontal morphism is obviously a quasi-isomorphism, and this makes the induced morphism $\Hom_{R}(A, V)\rightarrow \Hom_{R}(A, W)$ a quasi-isomorphism as well. \end{proof}

\begin{lemma}\label{lem2.25} If $X\simeq A$, where $A\in \C_{\sqsupset}(R)$ is a complex of Ding projective modules, and $U\simeq V$, where $V\in \C_{\sqsubset}(R)$ is a complex in which each module has finite flat or finite injective dimension, then \[\RHom_{R}(X, U)\simeq \Hom_{R}(A, V).\]
\end{lemma}
\begin{proof}Assume that $V\stackrel{\simeq}\rightarrow I\in \C_{\sqsubset}(R)$ is a DG-injective resolution of $V$. We have
\[\RHom_{R}(X, U)\simeq \RHom_{R}(A, V)\simeq \Hom_{R}(A, I).\]
From Lemma \ref{lem2.24} we get a quasi-isomorphism $\Hom_{R}(A, V)\stackrel{\simeq}\longrightarrow \Hom_{R}(A, I)$, and the result follows.
\end{proof}

\begin{lemma}\label{lem2.26} Let $F$ be a flat $R$-module. If $X\simeq A$, where $X\in \C_{(\square)}(R)$ and $A\in \C_{\sqsupset}(R)$ is a complex of Ding projective modules and $n\geq \sup X$, then
\[\Ext_{R}^{m}(C_{n}^{A}, F)=\h_{-(m+n)}(\RHom_{R}(X, F)).\]
\end{lemma}
\begin{proof}Since $n\geq\sup X=\sup A$ we have $A_{n}\sqsupset~\simeq \Sigma^{n}\C_{n}^{A}$, and since $F$ is flat it follows by Lemma \ref{lem2.25} that $\RHom_{R}(\C_{n}^{A}, F)$ is represented by $\Hom_{R}(\Sigma^{-n}A_{n}\sqsupset~, F)$.
For $m>0$ the isomorphism class $\Ext_{R}^{m}(\C_{n}^{A}, F)$ is then represented by
\[\begin{aligned}
\h_{-m}(\Hom_{R}(\Sigma^{-n}A_{n}\sqsupset, F))&=\h_{-m}(\Sigma^{n}\Hom_{R}(A_{n}\sqsupset, F))\\
&=\h_{-(m+n)}(\Hom_{R}(A_{n}\sqsupset, F))\\
&=\h_{-(m+n)}(\sqsubset_{-n}\Hom_{R}(A, F))\\
&=\h_{-(m+n)}(\Hom_{R}(A, F)).\\
\end{aligned}.\]
It also follows from Lemma \ref{lem2.25} that the complex $\Hom_{R}(A, F)$ represents $\RHom_{R}(X, F)$, so $\Ext_{R}^{m}(\C_{n}^{A}, F)=\h_{-(m+n)}(\RHom_{R}(X, F))$.
\end{proof}

\begin{theorem} \label{the2.27} Let $X\in \C_{(\sqsupset)}(R)$ of finite Ding projective dimension. For $n\in \mathbb{Z}$ the following are equivalent:

$\mathrm{(1)}$ $\Dpd_{R}(X)\leq n$.

$\mathrm{(2)}$ $\inf U-\inf \RHom_{R}(X, U)\leq n$ for all $U\in \C_{(\square)}(R)$ of finite flat dimension with $\h(U)\neq0$.

$\mathrm{(3)}$ $-\inf \RHom_{R}(X, F)\leq n$ for all flat $R$-modules $F$.

$\mathrm{(4)}$ $\sup X\leq n$ and for any bounded below complex $A\simeq X$ of Ding projective modules, the cokernel $\C_{n}^{A}=\Coker(A_{n+1}\rightarrow A_{n})$ is a Ding projective module.

Moreover, the following hold:

\[\begin{aligned}
\Dpd_{R}(X)&=\sup\{\inf U-\inf \RHom_{R}(X, U)~|~\fd_{R}U<\infty \ \text{and} \ \h(U)\neq0\}\\
&=\sup\{-\inf \RHom_{R}(X, F)~|~F \ \text{is flat}\}.\\
\end{aligned}.\]
\end{theorem}
\begin{proof}The proof of the equivalence of (1)-(4) is cyclic. Obviously (2)$\Rightarrow$ (3). So this leaves us three implications to prove.

(1)$\Rightarrow$ (2) Choose a complex $A\in \C_{\square}(R)$ consisting of Ding projective modules, such that $A\simeq X$ and $A_{l}=0$ for all $l>n$. First let $U$ be a complex of finite flat dimension with $\h(U)\neq 0$. Set $i=\inf U$ and note that $i\in \mathbb{Z}$ as $U\in \C_{(\square)}(R)$ with $\h(U)\neq 0$. Choose a bounded complex $F\simeq U$ of flat modules with $F_{l}\neq 0$ for $l<i$. By Lemma \ref{lem2.25}, the complex $\Hom_{R}(A, F)$ is equivalent to $\RHom_{R}(X, U)$; In particular, $\inf \RHom_{R}(X, U)=\inf \Hom_{R}(A, F)$. For $l<i-n$ and $q\in \mathbb{Z}$, either $q>n$ or $q+l\leq n+l<i$, so the module
\[\Hom_{R}(A, F)_{l}=\prod _{q\in \mathbb{Z}}\Hom_{R}(A_{q}, F_{q+l})=0.\]
Hence, $\h_{l}(\Hom_{R}(A, F))=0$ for $l<i-n$, and $\inf \RHom_{R}(X, U)\geq i-n=\inf U-n$ as desired.

(3)$\Rightarrow$ (4) This part is divided into three steps. First we establish the inequality $n\geq \sup X$, next we prove that the $n$th cokernel in a bounded complex $A\simeq X$ of Ding projective modules is again Ding projective, and finally we give an argument that allows us to conclude the same for $A\in \C_{\sqsupset}(R)$.

To see that $n\geq \sup X$, it is sufficient to show that
\begin{equation}\sup \{-\inf \RHom_{R}(X, F)|~F~ \text{is flat}\}\geq\sup X.\tag{$*$}\end{equation}
By assumption, $g=\Dpd_{R}X$ is finite; That is, $X\simeq A$ for some complex \[A=0\longrightarrow A_{g}\stackrel{\delta_{g}^{A}}\longrightarrow A_{g-1}\longrightarrow\cdots\longrightarrow A_{i}\longrightarrow0,\]
and it is clear $g\geq \sup X$ since $X\simeq A$. For any flat module $F$, the complex $\Hom_{R}(A, F)$ is concentrated in degrees $-i$ to $-g$,
\[\xymatrix{0\ar[r] &\Hom_{R}(A_{i}, F)\ar[r]&\cdots\ar[r]& \Hom_{R}(A_{g-1}, F)\ar[rr]^{\Hom_{R}(\delta_{g}^{A}, F)}&& \Hom_{R}(A_{g}, F)\ar[r]& 0}.\]
By lemma \ref{lem2.25}, $\Hom_{R}(A, F)$ is equivalent to $\RHom_{R}(X, F)$ in $\C(\mathbb{Z})$. First, consider the case $g=\sup X$: The differential $\delta_{g}^{A}: A_{g}\rightarrow A_{g-1}$ is not injective, as $A$ has homology in degree $g=\sup X=\sup A$. By the definition of Ding projective modules, there exists a projective (and so flat) module $F$ and an injective homomorphism $\varphi: A_{g}\rightarrow F$. Because $\delta_{g}^{A}$ is not injective, $\varphi\in \Hom_{R}(A_{g}, F)$ cannot have the form $\Hom_{R}(\delta_{g}^{A}, F)(\psi)=\psi\delta_{g}^{A}$ for some $\psi\in \Hom_{R}(A_{g-1}, F)$. That is, the differential $\Hom_{R}(\delta_{g}^{A}, F)$ is not surjective; Hence $\Hom_{R}(A, F)$ has nonzero homology in degree $-g=-\sup X$, and $(*)$ follows.

Next, assume that $g>\sup X=s$ and consider the exact sequence
\[0\longrightarrow A_{g}\longrightarrow \cdots\longrightarrow A_{s+1}\longrightarrow A_{s}\longrightarrow A_{s}\longrightarrow \C_{s}^{A}\longrightarrow 0.\]
It shows that $\Dpd_{R}\C_{s}^{A}\leq g-s$, and it is easy to check that equality must hold; otherwise, we would have $\Dpd_{R}X<g$ by \ref{prop2.22}. By Lemma \ref{lem2.26}, it follows that for all $m>0$, all $n\geq\sup X$, and all flat modules $F$ one has
\begin{equation}\Ext_{R}^{m}(\C_{n}^{A}, F)=\h_{-(m+n)}(\RHom_{R}(X, F)).\tag{$\star$}
\end{equation}
By Lemma \ref{lem2.7}, we have $\Ext_{R}^{g-s}(\C_{s}^{A}, F)\neq 0$ for some flat $F$, whence $\h_{-g}(\RHom_{R}(X, F))\neq 0$ by $(\star)$, and $(*)$ follows. We conclude that $n\geq\sup X$.

It remains to prove that $\C_{n}^{A}$ is Ding projective for any bounded below complex $A\simeq X$ of Ding projective modules. By assumption, $\Dpd_{R}X$ is finite, so a bounded complex $\widetilde{A}\simeq X$ of Ding projective modules does exist. Consider the cokernel $\C_{n}^{\widetilde{A}}$. Since $n\geq\sup X=\sup \widetilde{A}$, it fits in an exact sequence $0\rightarrow \widetilde{A_{t}}\rightarrow\cdots\rightarrow \widetilde{A_{n+1}}\rightarrow \widetilde{A_{n}}\rightarrow \C_{n}^{\widetilde{A}}\rightarrow 0$, where all the $\widetilde{A_{l}}$'s are Ding projective. By $(\star)$ and Lemma \ref{lem2.4}(3), it now follows that also $\C_{n}^{\widetilde{A}}$ is Ding projective. With this, it is sufficient to prove the following:

If $P, A\in \C_{\sqsupset}(R)$ are complexes of, respectively, projective and Ding projective modules, and $P\simeq X\simeq A$, then the cokernel $\C_{n}^{P}$ is Ding projective if and only if $\C_{n}^{A}$ is so.

Let $A$ and $P$ be two such complexes. As $P$ consists of projective modules, there is a quasi-isomorphism $\pi: P\stackrel{\simeq}\longrightarrow A$, which induces a quasi-isomorphism between the truncated complexes, $\subset_{ n}\pi: \subset_{ n}P\stackrel{\simeq}\longrightarrow \subset_{ n}A$. The mapping cone
\[\Cone(\subset_{ n}\pi)=0\rightarrow \C_{n}^{P}\rightarrow P_{n-1}\oplus \C_{n}^{A}\rightarrow P_{n-2}\oplus A_{n-1}\rightarrow \cdots \]
is a bounded exact complex, in which all modules but the two left-most ones are known to be Ding projective modules. It follows by the resolving properties of the class of Ding projective modules that $\C_{n}^{P}$ is Ding projective if and only if $P_{n-1}\oplus \C_{n}^{A}$ is so, which is equivalent to $\C_{n}^{A}$ being Ding projective.

(4)$\Rightarrow$ (1) Choose a DG-projective resolution $P$ of $X$, by (4) the truncation $\subset_{ n}P$ is a complex of the desired type.

The last claim are immediate consequences of the equivalence of (1), (2) and (3).
\end{proof}

\begin{corollary} Let $X\in \C_{(\sqsupset)}(R)$ of finite Ding projective dimension. Then $\Gpd_{R}(X)=\Dpd_{R}(X)$.
\end{corollary}

\begin{proof} It follows from Theorem \ref{the2.27} and \cite[Theorem 3.1]{Christensen2006}
\end{proof}

\begin{lemma}\label{lem2.27} Let $\varphi: R\rightarrow S$ be a homomorphism of rings.

$\mathrm{(1)}$ If $M$ is a Ding projective $S$-module, then $\Hom_{R}(\widetilde{P}, M)$ is a Ding projective $S$-module for every finite projective $R$-module $\widetilde{P}$.

$\mathrm{(2)}$ If $M$ is a Ding projective $S$-module, then $ \widetilde{P}\otimes_{R} M$ is a Ding projective $S$-module for every projective $R$-module $\widetilde{P}$.
\end{lemma}

\begin{proof}(1) Let $T$ be a totally $\mathcal{F}$-acyclic complex of $M$. Then the complex $\Hom_{R}(\widetilde{P}, T)$ of projective $S$-modules is exact. For any flat $S$-module $Q$ we have
\[\Hom_{S}(\Hom_{R}(\widetilde{P}, T), Q))\cong \widetilde{P}\otimes_{R}\Hom_{S}(T, Q).\]
Since $\Hom_{S}(T, Q)$ is exact, and $\widetilde{P}$ is finite projective, we obtain that $\Hom_{S}(\Hom_{R}(\widetilde{P}, T), Q))$ is exact, and so $\Hom_{R}(\widetilde{P}, T)$ is a totally $\mathcal{F}$-acyclic complex of $\Hom_{R}(\widetilde{P}, M)$.

(2) Let $T$ be a totally $\mathcal{F}$-acyclic complex of $M$. Then the complex $\widetilde{P}\otimes_{R}T$ of projective $S$-modules is exact. For any flat $S$-module $Q$ we have the exactness of
\[\Hom_{S}(\widetilde{P}\otimes_{R}T, Q))\cong \Hom_{R}(\widetilde{P}, \Hom_{S}(T, Q))\]
as $\Hom_{S}(T, Q)$ is exact. Hence $\widetilde{P}\otimes_{R}T$ is a totally $\mathcal{F}$-acyclic complex of $\widetilde{P}\otimes_{R} M$.
\end{proof}

\begin{theorem}\label{the2.28} Let $\varphi: R\rightarrow S$ be a homomorphism of rings, $X\in \C_{(\sqsupset)}(S)$.

$\mathrm{(1)}$ If $U\in \p^{(f)}(R)$, then $\Dpd_{S}(\RHom_{R}(U, X))\leq \Dpd_{S}(X)-\inf U$.

$\mathrm{(2)}$ If $U\in \p(R)$, then $\Dpd_{S}(U\otimes_{R}^{L}X)\leq \Dpd_{S}(X)+\pd_{R}(U)$.
\end{theorem}

\begin{proof} (1) We can assume that $U$ is not exact, otherwise the inequality is trivial; and we set $i=\inf U$, $\pd_{R}(U)=n$. The inequality is also trivial if $X$ is exact or not of finite Ding projective dimension, so we assume that $X$ is not exact and set $\Dpd_{R}(X)=g$. We can now choose a complex $A\in \C_{\square}(S)$ of Ding projective modules which is equivalent to $X$ and has $A_{l}=0$ for $l>g$; we set $v=\inf\{l\in\mathbb{Z}~|~A_{l}\neq0\}$. Since $U\in \p^{(f)}(R)$, $U$ is equivalent to a complex $P$ of finite projective modules concentrated in degrees $n$, $\cdots$, $i$. Now $\RHom_{R}(U, X)$ is represented by the complex $\Hom_{R}(P, A)$ with
\[\Hom_{R}(P, A)_{l}=\prod_{q\in\mathbb{Z}}\Hom_{R}(P_{q}, A_{q+l})=\bigoplus_{q=i}^{n}\Hom_{R}(P_{q}, A_{q+l}).\]
The modules $\Hom_{R}(P_{q}, A_{q+l})$ are Ding projective by Lemma \ref{lem2.27}(1), and finite direct sums of Ding projective modules are Ding projective. So $\Hom_{R}(P, A)$ is a complex of Ding projective modules. In addition, we have $\Hom_{R}(P, A)_{l}=0$ for $l<v-n$; and if $l>g-i$, then $l+q>g-i+q\geq g$, and so $\Hom_{R}(P, A)_{l}=0$. It follows that $\Hom_{R}(P, A)$ is bounded. That is, $\Hom_{R}(P, A)$ is a bounded complex of Ding projective modules concentrated in degrees at most $g-i$. Hence $\Dpd_{S}(\RHom_{R}(U, X))\leq g-i= \Dpd_{S}(X)-\inf U$.

(2) We can assume that $U$ is not exact, otherwise the inequality is trivial; and we set $i=\inf U$, $\pd_{R}(U)=n$. The inequality is also trivial if $X$ is exact or not of finite Ding projective dimension, so we assume that $X$ is not exact and set $\Dpd_{R}(X)=g$. We can now choose a complex $A\in \C_{\square}(S)$ of Ding projective modules which is equivalent to $X$ and has $A_{l}=0$ for $l>g$; we set $v=\inf\{l\in\mathbb{Z}~|~A_{l}\neq0\}$. Since $U\in \p(R)$, $U$ is equivalent to a complex $P$ of projective modules concentrated in degrees $n$, $\cdots$, $i$. Now $U\otimes_{R}^{L}X$ is represented by the complex $P\otimes_{R}A$ with
\[(P\otimes_{R}A)_{l}=\bigoplus_{q\in\mathbb{Z}}(P_{q}\otimes A_{l-q})=\bigoplus_{q=i}^{n}(P_{q}\otimes A_{l-q}).\]
The modules $P_{q}\otimes A_{l-q}$ are Ding projective by Lemma \ref{lem2.27}(2), and finite direct sums of Ding projective modules are Ding projective. So $P\otimes_{R}A$ is a complex of Ding projective modules. In addition, we have $(P\otimes_{R}A)_{l}=0$ for $l<v+i$; and if $l>g+n$, then $l-q>g+n-q\geq g$, and so $(P\otimes_{R}A)_{l}=0$. It follows that $P\otimes_{R}A$ is bounded. That is, $P\otimes_{R}A$ is a bounded complex of Ding projective modules concentrated in degrees at most $g+n$. Hence $\Dpd_{S}(U\otimes_{R}^{L}X)\leq g+n= \Dpd_{S}(X)+\pd_{R}(U)$.
\end{proof}

 Let $R$ be a subring of the ring $S$, and assume that $R$ and
$S$ have the same unity 1. The ring $S$ is called an excellent extension of
$R$ if

(A) $S$ is a free normalizing extension of $R$ with a basis that includes 1;
that is, there is a finite subset $\{a_{1}, \cdots, a_{n}\}$ of $S$ such that $a_{1}=1$, $S=\sum_{i=1}^{n}a_{i}R$ and $a_{i}R = Ra_{i}$ for all $i=1,\cdots, n$ and $S$ is free with basis
$\{a_{1}, \cdots, a_{n}\}$ both as a right and left $R$-module, and

(B) $S$ is $R$-projective; that is, if $_{S}N$ is a submodule of $_{S}M$, then $_{R}N~|~_{R}M$ implies $_{S}N~|~_{S}M$.

Excellent extensions were introduced by Passman \cite{Passman77}. Examples include $n\times n$ matrix rings, and crossed products $R*G$ where $G$ is a finite group with $|G|^{-1}\in R$ \cite{Passman83}.
\begin{lemma}\label{lem2.28} Let $S$ be an excellent extension of $R$.

$\mathrm{(1)}$  If $N$ is a Ding projective $R$-module, then $\Hom_{R}(\widehat{P}, N)$ is a Ding projective $S$-module for every finite projective $S$-module $\widehat{P}$.

$\mathrm{(2)}$ If $N$ is a Ding projective $R$-module, then $ \widehat{P}\otimes_{R} N$ is a Ding projective $S$-module for every projective $S$-module $\widehat{P}$.
\end{lemma}
\begin{proof}(1) Let $T$ be a totally $\mathcal{F}$-acyclic complex of $N$. Then the complex $\Hom_{R}(\widehat{P}, T)$ consists of projective $S$-modules, and it is exact as $\widehat{P}$ is a finite projective $R$-module. For any flat $S$-module $Q$ we have
\[\Hom_{S}(\Hom_{R}(\widehat{P}, T), Q))\cong \widehat{P}\otimes_{R}\Hom_{S}(T, Q),\]
which is exact as $Q$ is a flat $R$-module. Hence $\Hom_{R}(\widehat{P}, T)$ is a totally $\mathcal{F}$-acyclic complex of $\Hom_{R}(\widehat{P}, N)$.

(2) Let $T$ be a totally $\mathcal{F}$-acyclic complex of $N$. Then the complex $\widehat{P}\otimes_{R}T$ consists of projective $S$-modules, and it is exact as $\widehat{P}$ is a projective $R$-module. For any flat $S$-module $Q$, we have
\[\Hom_{S}(\widetilde{P}\otimes_{R}T, Q))\cong \Hom_{R}(\widetilde{P}, \Hom_{S}(T, Q)),\]
which is exact as $Q$ is a flat $R$-module. Hence $\widehat{P}\otimes_{R}T$ is a totally $\mathcal{F}$-acyclic complex of $\widehat{P}\otimes_{R} N$.
\end{proof}
\begin{theorem}Let $S$ be an excellent extension of $R$, $X\in \C_{(\sqsupset)}(R)$.

$\mathrm{(1)}$ If $V\in \p^{(f)}(S)$, then $\Dpd_{S}(\RHom_{R}(V, X))\leq \Dpd_{R}(X)-\inf V$.

$\mathrm{(2)}$ If $V\in \p(S)$, then $\Dpd_{S}(V\otimes_{R}^{L}X)\leq \Dpd_{R}(X)+\pd_{S}(V)$.
\end{theorem}
\begin{proof} They follow by analogy with the proof of Theorem \ref{the2.28}, only this time use Lemma \ref{lem2.28}.
\end{proof}

\begin{corollary}Let $S$ be an excellent extension of $R$, $X\in \C_{(\sqsupset)}(R)$. Then

$\mathrm{(1)}$ $\Dpd_{S}(\RHom_{R}(S, X))\leq \Dpd_{R}(X)$.

$\mathrm{(2)}$ $\Dpd_{S}(S\otimes_{R}^{L}X)\leq \Dpd_{R}(X)$.
\end{corollary}

In \cite{Mao2008}, Mao and Ding introduced and studied Gorenstein FP-injective modules, and showed that there is a very close relationship between Gorenstein FP-injective modules and Gorenstein flat modules. A left $R$-module $N$ is called Gorenstein FP-injective if there is an exact sequence
\[\cdots \longrightarrow E_{1}\longrightarrow E_{0}\longrightarrow E_{-1}\longrightarrow E_{-2}\longrightarrow \cdots\]
of injective left $R$-modules with $N=\Coker (E_{0}\longrightarrow E_{-1})$ such that $\Hom (E, -)$ leaves the sequence exact whenever $E$ an FP-injective $R$-module. Since Gorenstein FP-injective modules have properties analogous to Gorenstein injective modules, Gillespie \cite{Gillespie2010} called these modules Ding injective modules.

\begin{remark}Above we have only mentioned the Ding projective dimension of $R$-complexes. Dually one can also define and study the Ding injective dimension for complexes of left $R$-modules over an associative ring $R$. All the results concerning Ding projective dimension have a Ding injective counterpart.
\end{remark}
\vspace{0.3cm} \hspace{-0.8cm}{\bf{Acknowledgements}}

The authors would like to thank the referee for valuable suggestions
and helpful corrections.

\end{document}